\documentclass{amsart}

\usepackage{color,graphicx,amssymb,latexsym,amsfonts,txfonts,amsmath,amsthm}
\usepackage{pdfsync}
\usepackage{amsmath,amscd}
\usepackage[all,cmtip]{xy}

\usepackage{tikz}
\usetikzlibrary{matrix}

\input epsf

\usepackage{hyperref}
\hypersetup{
    colorlinks=true,       
    linkcolor=blue,          
    citecolor=blue,        
    filecolor=blue,      
    urlcolor=blue           
}

\theoremstyle{plain}
\newtheorem{theo}{Theorem}

\newtheorem{lemm}{Lemma}
\newtheorem{coro}{Corollary}

\theoremstyle{definition}

\newtheorem{rema}{Remark}

\begin{document}

\title{Cyclic-Schottky strata of Schottky space}
\author{Rub\'en A. Hidalgo}
\address{Departamento de Matem\'atica y Estad\'{\i}stica, Universidad de La Frontera. Temuco, Chile}
\email{ruben.hidalgo@ufrontera.cl}

\author{Milagros Izquierdo}
\address{Mathematiska Institutionen, Link\"opings Universitet, 581 83 Link\"opings, Sweden}
\email{milagros.izquierdo@liu.se}

\thanks{Partially supported by Projects Fondecyt 1230001 and 1220261}


\begin{abstract} 
Schottky space  ${\mathcal S}_{g}$, where $g \geq 2$ is an integer, is a connected complex orbifold of dimension $3(g-1)$; it provides a parametrization 
of the ${\rm PSL}_{2}({\mathbb C})$-conjugacy classes of Schottky groups $\Gamma$ of rank $g$.  The 
branch locus ${\mathcal B}_{g} \subset {\mathcal S}_{g}$, consisting of those conjugacy classes of  Schottky groups being a finite index proper normal subgroup of some Kleinian group, is known to be connected. If $[\Gamma] \in {\mathcal B}_{g}$, then there is a Kleinian group $K$ containing $\Gamma$ as a normal subgroup of index some prime integer $p \geq 2$.
The structural description, in terms of Klein-Maskit Combination Theorems, of such a group $K$  is completely determined by a triple $(t,r,s)$, where $t,r,s \geq 0$ are integers such that $g=p(t+r+s-1)+1-r$. For each such a tuple $(g,p;t,r,s)$  there is a corresponding cyclic-Schottky stratum $F(g,p;t,r,s) \subset {\mathcal B}_{g}$. 
It is known that $F(g,2;t,r,s)$ is connected. 
In this paper, for $p \geq 3$, we study the connectivity of these $F(g,p;t,r,s)$. 
\end{abstract}

\maketitle

\section{Introduction}
Let $\Gamma$ be a Schottky group of rank $g \geq 2$ (a purely loxodromic Kleinian group, isomorphic to the free group $F_{g}$ of rank $g$, and with a non-empty region of discontinuity). Its region of discontinuity $\Omega$ is non-empty and connected (the complement of a Cantor set) and 
$S=\Omega/\Gamma$ is a closed Riemann surface of genus $g$ (we say that $\Gamma$ uniformizes $S$). 
As a consequence of Koebe's retrosection theorem \cite{Bers,Koebe}, every closed Riemann surface of genus $g \geq 2$ is uniformized by a suitable Schottky group of rank $g$. Moreover, by the planarity theorem \cite{Maskit:planarity}, Schottky groups correspond to their lowest uniformizations.

The quasiconformal deformation space 
${\mathcal Q}(\Gamma)$ is a complex manifold of dimension $3(g-1)$ \cite{Bers, Nag} and 
its group of holomorphic automorphisms is isomorphic to ${\rm Out}(F_{g})$ \cite{Earle}. As ${\rm Out}(F_{g})$ acts discontinuously on it, the quotient orbifold ${\mathcal Q}(\Gamma)/{\rm Out}(F_{g})$ is a complex orbifold of dimension $3(g-1)$. This quotient orbifold can be identified with the Schottky space ${\mathcal S}_{g}$, that parametrizes the ${\rm PSL}_{2}({\mathbb C})$-conjugacy classes of Schottky groups of rank $g$.

Inside ${\mathcal S}_{g}$ is its branch locus  ${\mathcal B}_{g}$, which consists of the (conjugacy classes of) Schottky groups which are a proper finite index normal subgroup of some Kleinian group. Note that we may assume, without loss of generality, that the index is a prime integer.
If $g \geq 3$, then ${\mathcal B}_{g}$ is exactly the locus of ${\mathcal S}_{g}$ where it fails to be a topological manifold. If $g=2$, then ${\mathcal B}_{2}={\mathcal S}_{2}$. This last fact comes from the observation that every rank two Schottky group $G=\langle A, B\rangle$  is an index two subgroup of $K=\langle E=AB-BA, EA, EB\rangle \cong {\mathbb Z}_{2} * {\mathbb Z}_{2} * {\mathbb Z}_{2}$ \cite{Keen}. In \cite{H-I, H-I:2} it was proved that ${\mathcal B}_{g}$ is always connected.

We will say that a tuple $(g,p;t,r,s)$ is {\it admissible} if (i) $p \geq 2$ is a prime integer and (ii) $t,r,s \geq 0$ are integers such that $g=p(t+r+s-1)+1-r$.

Let $K$ be a  Kleinian group admitting a Schottky group $\Gamma$ of rank $g$ as a normal subgroup of index a prime integer $p \geq 2$. In \cite{Hidalgo:SchottkyAuto} it was observed that the geometrical structure of $K$, in terms of Klein-Maskit Combination Theorems  \cite{Maskit:Comb,Maskit:Comb4}, is uniquely determined by an admissible tuple $(g,p;t,r,s)$ (see Section \ref{Sec:cyclicschottky}). In this case, we say that $K$  is a {\it cyclic-Schottky group} of type $(g,p;t,r,s)$. This structure description, for example, permits us to observe that: (i) any two cyclic-Schottky groups of the same type are quasiconformally conjugated, (ii) any admissible tuple is the type of a cyclic-Schottky group, and (iii)
if $K$ is a cyclic-Schottky group of type $(g,p;t,r,s)$, with region of discontinuity $\Omega$,  and $\Gamma$ is a Schottky group inside $K$ as a normal subgroup of index $p$, then $S=\Omega/\Gamma$ is a closed Riemann surface of genus $g$ admitting a conformal automorphism $\tau$ of order $p$ such that $S/\langle \tau \rangle= \Omega/K$ is an orbifold of genus $t+s$ and with exactly $2r$ cone points, each one of oder $p$.

Conversely to (iii) above, assume that $S$ is a closed Riemann surface of genus $g$ admitting a conformal automorphism $\tau$ of order $p$ such that $S/\langle \tau \rangle$ has genus $t+s$ and exactly $2r$ cone points of order $p$. If there is a Schottky group $\Gamma$ uniformizing $S$ for which $\tau$ lifts, then there is a cyclic-Schottky group $K$ of some type $(g,p;t,r,s)$ containing $\Gamma$ as a normal subgroup of index $p$.  In this way, cyclic-Schottky groups of type $(g,p;t,r,s)$ correspond to the Schottky uniformizations of a genus $g$ Riemann surface reflecting an order $p$ conformal automorphism with quotient orbifold as above. A $3$-dimensional context (at the level of automorphisms of handlebodies) can be found in \cite{Zimmermann}.

Associated to an admissible tuple $(g,p;t,r,s)$ is the cyclic-Schottky strata $F(g,p;t,r,s) \subset {\mathcal B}_{g}$, formed by the (conjugacy classes of) Schottky groups of rank $g$ which are contained as an index $p$ normal subgroup of some cyclic-Schottky group of type $(g,p;t,r,s)$. The branch locus ${\mathcal B}_{g}$ is the finite union of those strata $F(g,p;t,r,s)$, where $(g,p;t,r,s)$ runs over all possible admissible tuples. 

Each cyclic-Schottky stratum $F(g,p;t,r,s)$ is a finite union of connected complex orbifolds (which might or not intersect), called its {\it irreducible components}.  Each irreducible component is isomorphic to the complex orbifold ${\mathcal Q}(K)/{\rm Mod}(K)$, where ${\mathcal Q}(K)$ is the quasiconformal deformation space of a cyclic-Schottky group $K$ of type $(g,p;t,r,s)$  and ${\rm Mod}(K)$ is its modular group. In particular, any two irreducible components of the same $F(g,p;t,r,s)$ are isomorphic complex orbifolds of dimension $(3g-3-r(p-3))/p$ (the dimension of ${\mathcal Q}(K)$, see Remark \ref{obs:dimension}). The maximal dimension is obtained for $p=2$, $r=g+1$, $t=s=0$;  in this case, $F(g,2;0,g+1,0)$ is the locus of classes of hyperelliptic Schottky groups \cite{Keen}.

In Theorem \ref{topo}, we provide the number of irreducible components of $F(g,p;t,r,s)$. This, in particular, gives us upper bounds for the number of its connected components. As different irreducible components might intersect, such an upper bound could be bigger than the number of its connected components (see Theorem \ref{maintheo}).

As a matter of completeness, in Section \ref{Sec:consecuencias}, we describe how to count the number of different cyclic-Schottky strata for fixed genus $g$. Explicit formulae are given for $p \in \{2,3\}$ and a short algorithm is provided for a general situation.

\section{Main results}
Before stating our main results, we need to recall some definitions. Let $K$ be a Kleinian group.
By a {\it geometric automorphism} of $K$ we mean a quasiconformal homeomorphism of the Riemann sphere $\widehat{\mathbb C}$ that self-conjugates it. Two subgroups of $K$ are called {\it geometrically equivalent in $K$} if there exists a geometric automorphism of $K$ conjugating them.

Let us now assume that $K$ is a cyclic-Schottky group of type $(g,p;t,r,s)$ (recall that $p \geq 2$ is always assumed to be a prime integer). 
If $p=2$, then in \cite{DGGH} it was proved that any two Schottky subgroups of index two of $K$ are geometrically equivalent. 
Our first result provides the number of geometrical equivalence classes of Schottky normal subgroups of index $p$ in $K$.

\begin{theo}\label{topo}
Let $(g,p;t,r,s)$ be an admissible tuple and let $K$ be a cyclic-Schottky group of type $(g,p;t,r,s)$. Then 
the number of index $p$ normal subgroups of $K$, up to geometric automorphisms of $K$, which are Schottky groups  (necessarily of rank $g$)
 is equal to  
$$M(g,p;t,r,s)=\left\{\begin{array}{cl}
1 & ; p=2.\\
\\
\left( \begin{array}{c}
r+(p-3)/2 \\
(p-3)/2
\end{array}
\right)
\left( \begin{array}{c}
s+(p-3)/2 \\
(p-3)/2
\end{array}
\right) & ; p \geq 3.
\end{array}
\right.
$$
\end{theo}

\begin{rema}\label{obs1}
Note that $M(g,p;t,r,s)$ is independent of $t$.  If either (i) $p=2,3$ or (ii) $p \geq 5$ and $r=s=0$, then a cyclic-Schottky group of type $(g,p;t,r,s)$ has exactly one Schottky subgroup of index $p$, up to geometrical equivalence.  An interpretation of $M(g,p;t,r,s)$, in terms of ${\mathbb Z}_{p}$-covers of handlebodies, is given in Remark \ref{unicidad}.
\end{rema}

As a consequence of Theorem \ref{topo} and Remark \ref{obs1} one has the following.

\begin{coro}
$F(g,p;t,r,s)$ has exactly $M(g,p;t,r,s)$ irreducible components, in particular, at most $M(g,p;t,r,s)$ connected components.
\end{coro}

Two different irreducible components of $F(g,p;t,r,s)$ may intersect. We provide an example of intersection in Example \ref{ejemplo2}.
This implies that the number of connected components could be strictly smaller than $M(g,p;t,r,s)$.
Our second result concerns the connectivity of these cyclic-Schottky groups' strata and the number of their connected components.

\begin{theo}\label{maintheo} 
Let $(g,p;t,r,s)$ be an admissible tuple.

\noindent
(1) If either (i) $p=2,3$ or (ii) $p \geq 5$ and $r=s=0$, then $F(g,p;t,r,s)$ is connected.

\noindent
(2) If $p \geq 5$, then $F(g,p;t,r,s)$ has at most $M(g,p;t,r,s)$ connected components.

\noindent
(3) If $p \geq 5$ and either 
(i) $r \nequiv 0\; {\rm mod}(p)$ or (ii) $s \nequiv 0\; {\rm mod}(p)$,
then $F(g,p;t,r,s)$ is not connected and it has exactly $M(g,p;t,r,s)$ connected components.

\end{theo}

\begin{rema}\label{obs2}
\begin{enumerate}
\item[(a)] The connectivity, for $p=2$, was proved in \cite{DGGH}.
\item[(b)] If $p \geq 5$ and $g \nequiv 1\; {\rm mod}(p)$, then $r \nequiv 0\; {\rm mod}(p)$; so part (3) of Theorem \ref{maintheo} asserts that $F(g,p;t,r,s)$ consists of exactly $M(g,p;t,r,s)$ connected components, that is, irreducible components are pairwise disjoint.
\item[(c)] In \cite{L-H} it was observed that there is a prime integer $p_{0}$ (depending on the triple $(t,r,s)$) so that, if $p \geq p_{0}$ and $S$ is a closed Riemann surface admitting $H \cong {\mathbb Z}_{p}$ as a group of conformal automorphisms such that $S/H$ has genus $\gamma=t+s$ and $2r$ cone points, then $H$ is the unique $p$-Sylow subgroup of ${\rm Aut}(S)$. In this situation, $F(g,p;t,r,s)$ will consist of exactly $M(g,p;t,r,s)$ connected components. 
\end{enumerate}
\end{rema}

\section{Preliminaries}\label{prelim}
We use the symbol $\Gamma<K$ (respectively, $\Gamma \lhd K$) to say that $\Gamma$ is a subgroup (respectively, a normal subgroup) of a group $K$. We denote by ${\mathbb M} = {\rm PSL}_{2}({\mathbb C})$ the group of {\it M\"obius transformations} (the full group of conformal automorphisms of $\widehat{\mathbb C}$). 
Each M\"obius  transformation acts (by Poincar\'e's extension) as an orientation-preserving isometry of the hyperbolic $3$-space ${\mathbb H}^{3}=\{(z,t) \in {\mathbb C} \times {\mathbb R}: t>0\}$ with the hyperbolic metric $ds^{2}=(|dz|^2+dt^2)/t^2$. 
The composition of two maps $f$ and $h$ is, as usually, denoted by the symbol $f \circ h$, but if we are composing M\"obius transformations $A$ and $B$ we will write $AB$.

\subsection{Kleinian groups}
 A  {\it Kleinian group} is a discrete subgroup $K$ of ${\mathbb M}$. We say that 
$K<{\mathbb M}$ {\it acts discontinuously} at $p \in \widehat{\mathbb C}$ if there is an open neighborhood $U$ of $p$ such that,
up to finitely many elements $A \in K$, it holds that 
$A(U) \cap U = \emptyset$ (in particular, the $K$-stabilizer of $p$ is finite). The {\it region of discontinuity} of $K$ is the open set $\Omega \subset \widehat{\mathbb C}$ (which might be empty) consisting of those points on which $K$ acts discontinuously; its complement $\Lambda=\widehat{\mathbb C} \setminus\Omega$ is its {\it limit set} (if $\Lambda$ is finite, then $K$ is called {\it elementary}; otherwise {\it non-elementary}).  

\begin{rema}
If $K$ is a Kleinian group and $\Omega$ is its region of discontinuity, then $M_{K}=({\mathbb H}^{3} \cup \Omega)/K$ is a $3$-dimensional orbifold; its interior ${\mathbb H}^{3}/K$ carries a hyperbolic structure and, if $\Omega \neq \emptyset$, its conformal boundary $S_{K}=\Omega/K$ a Riemann orbifold structure (a Riemann surface with cone points). If $K$ is a torsion-free Kleinian group, then $M_{K}$ is a manifold (with boundary if $\Omega \neq \emptyset$) whose interior is a hyperbolic $3$-manifold, and $S_{K}$ is a Riemann surface (without cone points). If $K$ has a finite-sided fundamental polyhedron for its action on ${\mathbb H}^{3}$, then it is called {\it geometrically finite}.
\end{rema}

If $K_{1}<K_{2}<{\mathbb M}$ and $K_{1}$ have finite index in $K_{2}$, then one is discrete if and only if the other is (in which case both have the same region of discontinuity). There are examples of Kleinian groups with empty regions of discontinuity; for instance ${\rm PSL}_{2}({\mathbb Z}[i])$. 
In this paper, we will only consider Kleinian groups with a non-empty and connected region of discontinuity. Finitely generated Kleinian groups with an invariant connected component of the region of discontinuity are called {\it function groups} and their geometrical structure, in the sense of Klein-Maskit Combination Theorems, is provided in \cite{Maskit:function2, Maskit:function1, Maskit:function}. Generalities on  Kleinian groups can be found, for instance,  in the books \cite{Maskit:book, MT}.

\subsection{Schottky groups}\label{Sec:Schottkygroups}
Schottky groups are particular examples of Kleinian groups which are obtained from the Klein-Maskit Combinations Theorem by amalgamating several times cyclic groups generated by loxodromic transformations.

A {\it Schottky group of rank $g \geq 1$} is a Kleinian group $\Gamma$ generated by loxodromic
transformations $A_1,\ldots,A_g$, so that there are $2g$ pairwise disjoint simple loops,
$\Sigma_1,\Sigma'_1,\ldots,\Sigma_g, \Sigma'_g$, bounding a $2g$-connected domain ${\mathcal D}\subset \widehat{\mathbb C}$, where $A_i(\Sigma_i)=\Sigma'_i$, and $A_i({\mathcal D})\cap {\mathcal D}=\emptyset$, for $i=1,\ldots,g$. The collection of loops $\Sigma_{1}, \Sigma'_{1},\ldots,\Sigma_{g}, \Sigma'_{g}$ is called a {\it fundamental set of loops}  for $\Gamma$ with respect to the above generators. 
Its region of discontinuity $\Omega$ is connected (so $\Gamma$ is a function group) and dense in $\widehat{\mathbb C}$, the quotient $S_{\Gamma}=\Omega/\Gamma$ is a closed Riemann surface of genus $g$ (the classical retrosection theorem states that, up to conformal isomorphism, every closed Riemann surface of genus $g$ is obtained in this way). The  
$3$-manifold $M_{\Gamma}=({\mathbb H}^{3} \cup \Omega)/\Gamma$ is homeomorphic to a handlebody of genus $g$ (conversely, every torsion free Kleinian group $\Gamma$, for which $M_{\Gamma}$ is homeomorphic to a handlebody of genus $g$, is a Schottky group of rank $g$).

\begin{rema}
In \cite{Chuckrow}, Chuckrow proved that for any set of $g$ generators of a Schottky group of rank $g \geq 1$ there exists a corresponding fundamental set of loops.
In \cite{Maskit:Schottky groups}, Maskit proved that a Schottky group of rank $g$ is the same as a purely loxodromic Kleinian group, of the second, which isomorphic to a free group of rank $g$. 
From the above geometrical construction, any two Schottky groups of the same rank are quasiconformally conjugated.
\end{rema}

\subsection{Cyclic-Schottky groups}\label{Sec:cyclicschottky}
A {\it $p$-cyclic-Schottky group} is a Kleinian group $K$ containing  a Schottky group $\Gamma$
as a normal subgroup such that $K/\Gamma$ is a cyclic group of order $p$. A geometrical structure picture (similar as for Schottky groups) of $p$-cyclic-Schottky groups, in terms of Klein-Maskit Combination Theorems,  was obtained in \cite{Hidalgo:CyclicSchottky}. Below we recall it for the case when $p \geq 2$ is a prime integer.

\begin{theo}[\cite{Hidalgo:CyclicSchottky}]\label{teo1}
Let $K$ be a $p$-cyclic-Schottky group, where $p \geq 2$ is a prime integer.
Then there is an admissible tuple $(g,p;t,r,s)$ such that $K$
can be constructed, by Klein-Maskit Combination Theorem, as the free product of ``$t$" 
cyclic groups, each one generated by a loxodromic transformation, ``$r$" cyclic groups, each one generated by an elliptic transformation of order $p$, and ``$s$" abelian groups isomorphic to $\mathbb{Z}\oplus \mathbb{Z}_{p}$,
each one generated by a loxodromic transformation and an elliptic transformation of
order $p$ (in this case, both share the same fixed points). In this case, we say that $K$ is a cyclic-Schottky group of type $(g,p;t,r,s)$.
\end{theo}

\begin{rema}[Uniqueness of the structural description]\label{unicidad}
Let $g$ and $p$ be fixed. If $K$ is a cyclic-Schottky group of type $(g,p;t,r,s)$, then the values of $t$, $r$, and $s$ are uniquely determined by $K$. In fact, if ${\mathcal C} \subset {\mathcal O}^{3}_{K}$ is the locus of cone points (the branch values of the natural quotient map ${\mathbb H}^{3} \to {\mathcal O}^{3}_{K}$), then the number of connected components of ${\mathcal C}$ which are simple loops is exactly $s$ and the number of connected components of ${\mathcal C}$ which are simple arcs is $r$. 
The number $t+s$ is the genus of the conformal boundary $\Omega/K$, where $\Omega$ is the region of discontinuity of $K$. In these terms, $M(g,p;t,r,s)$ is the number of ${\mathbb Z}_{p}$-covers of the orbifold handlebody ${\mathcal O}^{3}_{K}$, up to topological equivalence.
\end{rema}

\subsection{Quasiconformal homeomorphisms}
Let $\Omega_{1}, \Omega_{2} \subset \widehat{\mathbb C}$ be non-empty domains. An orientation-preserving  homeomorphism
$W:\Omega_{1} \to \Omega_{2}$ is called {\it quasiconformal} if it satisfies the following two conditions:

\noindent
(i) $W$ has distributional partial derivatives, with respect to $z$ and $\overline{z}$, which can be represented by locally integrable functions $W_{z}$ and $W_{\overline{z}}$ on $\Omega_{1}$; and

\noindent
(ii) there is a measurable function $\mu:\Omega_{1} \to {\mathbb C}$ (called a {\it complex dilation} of $W$) with $\mu \in L^{\infty}_{1}(\Omega_{1})$ (i.e., 
$\|\mu\|_{\infty}<1$, where $\| \mbox{ } \|_{\infty}$ denotes the essential supreme norm),  and $W$ satisfies the Beltrami equation
$$W_{\overline{z}} (z)= \mu(z) W_{z}(z) \quad a.e. \; z \in \Omega_{1}.$$

The existence and uniqueness of quasiconformal homeomorphisms is due to Morrey \cite{Morrey} and the continuous variation of the (normalized) solutions was shown by Ahlfors-Bers  \cite{A-B}.

\begin{theo}[\cite{A-B,Morrey}] \label{A-B}
\mbox{}

\noindent
(1) (Existence). 
If $\mu \in L^{\infty}_{1}({\mathbb C})$, then there is a unique 
quasiconformal homeomorphism $W_{\mu}:\widehat{\mathbb C} \to \widehat{\mathbb C}$, with complex dilation $\mu$, satisfying  $$W_{\mu}(\infty)=\infty, \; W_{\mu}(0)=0, \; W_{\mu}(1)=1.$$

\noindent
(2) (Measurable Riemann Mapping's Theorem).
If $\mu \in L^{\infty}_{1}(\Omega_{1})$ varies continuously (in the Banach space $L^{\infty}({\mathbb C})$), then $W_{\mu}$ varies locally uniformly continuously in the space of continuous functions on $\widehat{\mathbb C}$.

\end{theo}

Kleinian groups are called {\it (quasiconformally) topologically conjugated} if there is a (quasiconformal) homeomorphism of the Riemann sphere that conjugates one onto the other.

\subsection{Quasiconformal deformation spaces of Kleinian groups}
Let $K$ be a finitely generated, non-elementary Kleinian group with a region of discontinuity $\Omega \neq \emptyset$ (we will be mainly interested in the case when $K$ is either a Schottky group or a cyclic-Schottky group). We proceed to recall the definitions of the quasiconformal deformation space of $K$, its modular group, and its corresponding moduli space (the specialist might skip this section and just return to it, if necessary, for consulting notations).

\subsubsection{Beltrami coefficients of $K$}
 Associated to $K$ is the Banach space $L^{\infty}(K)$ (with the essential supreme norm $\| \mbox{ } \|_{\infty}$) whose elements are those measurable functions $\mu:\widehat{\mathbb C} \to {\mathbb C}$ so that 
$$\left\{ \begin{array}{ll}
\mu(z)=0, & \forall z \in \Lambda=\widehat{\mathbb C}-\Omega\\
\\
\mu(k(z))\overline{k_{z}(z)}=
k_{z}(z) \mu(z), & \mbox{if $z\in \Omega$ (a.e.) and $k \in K$}.
\end{array}
\right.
$$

Let $L^{\infty}_{1}(K)$ be the open unit ball in $L^{\infty}(K)$; its elements are called the {\it Beltrami coefficients} of $K$. 
By Theorem \ref{A-B}, for each Beltrami coefficient $\mu \in L^{\infty}_{1}(K)$ there is a unique quasiconformal homeomorphism $W_{\mu}:\widehat{\mathbb C} \to \widehat{\mathbb C}$, with complex dilation $\mu$, that fixes $0$, $1$ and $\infty$. Now, for each $k \in K$, the element $k_{\mu}=W_{\mu} \circ k \circ {W_{\mu}}^{-1}$ is again a M\"obius transformation. If we set $K_{\mu}=W_{\mu} K {W_{\mu}}^{-1}$, then the above provides an isomorphism of Kleinian groups $\chi_{\mu}:K \to K_{\mu}: k \mapsto k_{\mu}$ (the image $W_{\mu}(\Omega)$ is the region of discontinuity of $K_{\mu}$). 

\subsubsection{The quasiconformal deformation space of $K$}
Two Beltrami coefficients $\mu_{1}, \mu_{2} \in  L^{\infty}_{1}(K)$ are called {\it quasiconformal equivalent} (we denoted it by $\mu_{1} \sim \mu_{2}$), if $\chi_{\mu_{1}}=\chi_{\mu_{2}}$. As the group $K$ is non-elementary, this is equivalent to say that $W_{\mu_{1}}$ and $W_{\mu_{2}}$ coincide on the limit set $\Lambda$ of $K$. The quotient space  ${\mathcal Q}(K)=L^{\infty}_{1}(K)/\hspace{-0.2cm}\sim$ is called the {\it quasiconformal deformation space} of $K$.  As a consequence of the Measurable Riemann Mapping's Theorem, ${\mathcal Q}(K)$ is connected. As the Kleinian group $K$ has been assumed to be finitely generated, it is well known that ${\mathcal Q}(K)$ is a complex manifold of finite dimension \cite{Maskit:selfmaps}. 

\subsubsection{The modular group and moduli space of $K$}
If $A(K)$ denotes the group of quasiconformal homeomorphism $W:\widehat{\mathbb C} \to \widehat{\mathbb C}$ such that $W K W^{-1}=K$ (i.e., geometric automorphisms of $K$) and $A_{0}(K)$ is its normal subgroup of those automorphisms isotopic to the identity, then the quotient group ${\rm Mod}(K)=A(K)/A_{0}(K)$ is called the {\it modular group} of $K$.
It acts on the quasiconformal deformation space ${\mathcal Q}(K)$ by the rule
$${\rm Mod}(K) \times {\mathcal Q}(K) \to {\mathcal Q}(K):([W],[\mu]) \mapsto [\nu],$$
where $\nu$ is a complex dilation of the quasiconformal homeomorphism $W_{\mu} \circ W^{-1}$, where $W_{\mu}$ is a quasiconformal homeomorphism with complex dilation $\mu$. This action is known to be a discrete action by holomorphic automorphisms of the complex manifold ${\mathcal Q}(K)$. The quotient  orbifold ${\mathcal M}(K)={\mathcal Q}(K)/{\rm Mod}(K)$
is the {\it moduli space} of $K$. One may think of the moduli space being 
defined by forgetting the marking, that is, the space of ${\rm PSL}_{2}({\mathbb C})$-conjugate classes of quasiconformal deformations of $K$. We denote by $\pi_{K}:{\mathcal Q}(K) \to {\mathcal M}(K)$ the associated Galois quotient map (branched coverings) induced by the modular group.

\subsection{Quasiconformal deformation space of Schottky groups}\label{Schottky}
If $\Gamma$ is a Schottky group of rank $g \geq 2$, then its quasiconformal deformation space ${\mathcal Q}(\Gamma)$ is a complex manifold of dimension $3(g-1)$ and its moduli space ${\mathcal M}(\Gamma)$ is a complex orbifold of the same dimension \cite{B}. Next, we recall classic models for ${\mathcal Q}(\Gamma)$ and ${\mathcal M}(\Gamma)$
called the marked Schottky space and the Schottky space, respectively.

\subsubsection{\bf The marked Schottky space ${\mathcal MS}_{g}$}\label{identifica}
A {\it marked Schottky group of rank $g$} is a tuple $(\Gamma;A_{1}, \dots,A_{g})$, where $\Gamma$ is a Schottky group of rank $g$ and $A_{1}$,..., $A_{g}$ is a set of generators of $\Gamma$. Two marked Schottky groups of rank $g$, say  $(\Gamma; A_{1}, \dots,A_{g})$ and $(\widehat{\Gamma};\widehat{A}_{1},\dots,\widehat{A}_{g})$,  are said to be {\it equivalent} if there is a M\"obius transformation $T$ so that $T A_{j} T^{-1}=\widehat{A}_{j}$, for every $j=1,\dots,g$. We denote by $[(\Gamma;A_{1},\dots,A_{g})]$ the equivalence class of the marked Schottky group $(\Gamma;A_{1}, \dots, A_{g})$. 
The space that parametrizes equivalence classes of marked Schottky groups of rank $g$ is called the 
{\it marked Schottky space of rank $g$}, denoted by ${\mathcal MS}_{g}$. Note that ${\mathcal MS}_{1}$ can be identified with the punctured unit disc.

If $g \geq 2$, then ${\mathcal MS}_{g}$ has a natural structure of a complex manifold of dimension $3(g-1)$. This essentially comes from the following normalization choice. Given a marked Schottky group $(\Gamma; A_{1}, \dots,A_{g})$, we may find a M\"obius transformation $T$ so that the attracting fixed points of $A_{1}$, $A_{2}$ and $A_{2}A_{1}$ are respectively $\infty$, $0$ and $1$. Then the repelling fixed points of $A_{1},\ldots,A_{g}$, together with the attracting fixed points of $A_{3},\ldots,A_{g}$ and the repelling fixed points of $A_{2}A_{1},\ldots,A_{g}A_{1}$ provide a global coordinate. Let us fix some Schottky group $\Gamma$ of rank $g$ and a set of generators of it, say $A_{1}$,..., $A_{g}$. Then the map 
$$\chi:{\mathcal Q}(\Gamma) \to {\mathcal MS}_{g}: 
[\mu] \mapsto [(\Gamma_{\mu}=W_{\mu} \Gamma W_{\mu}^{-1}; W_{\mu} A_{1} W_{\mu}^{-1},\dots,W_{\mu} A_{g} W_{\mu}^{-1}]$$ 
turns out to be an isomorphism \cite{Bers}; so ${\mathcal MS}_{g}$ is a model of the quasiconformal deformation of a Schottky group of rank $g$.

Earle \cite{Earle} proved that for $g \geq 2$ the group of analytic automorphisms of ${\mathcal MS}_{g}$ is isomorphic to the group ${\rm Out}(F_{g})$ of outer automorphisms of the free group $F_{g}$ of rank $g$. The action of ${\rm Out}(F_{g})$ on $[\Gamma; A_{1},...,A_{g}]$ is just the change of the set of generators, that is, the ${\rm Out}(F_{g})$-orbit of $[\Gamma; A_{1},...,A_{g}]$ consists of all elements of the form $[\Gamma; B_{1},...,B_{g}]$. The stabilizer of $[\Gamma; A_{1},...,A_{g}]$ in ${\rm Out}(F_{g})$ can be identified with the set of M\"obius transformations $T$ normalizing $\Gamma$, in particular, with a Kleinian group $K=\langle T, \Gamma\rangle$ containing $\Gamma$ as a finite index normal subgroup.

\subsubsection{\bf The Schottky space and its branch locus}
If $\Gamma$ is a Schottky group of rank $g \geq 2$, then we denote by $[\Gamma]$ the equivalent class of ${\rm PSL}_{2}({\mathbb C})$-conjugates of $\Gamma$. The space that parametrizes conjugacy classes of Schottky groups of rank $g$ is called the {\it Schottky space of rank $g$}, which we denote by ${\mathcal S}_{g}$. This space is a model for the 
moduli space of a Schottky group of rank $g$ and it is a complex orbifold of dimension $3(g-1)$ for $g \geq 2$.
In the above explicit models, the natural holomorphic branched cover map $\pi_{\Gamma}: {\mathcal Q}(\Gamma) \to {\mathcal M}(\Gamma)$ whose deck group is ${\rm Mod}(\Gamma)={\rm Out}(F_{g})$ is given by the forgetting generators map 
$\pi:{\mathcal MS}_{g} \to {\mathcal S}_{g}: [(\Gamma; A_{1},\dots,A_{g})] \mapsto [\Gamma].$
The branch locus of $\pi$ is exactly ${\mathcal B}_{g}$ and it consists of those $[\Gamma] \in {\mathcal S}_{g}$ for which there is a Kleinian group $K$ containing $\Gamma$ as a proper finite index normal subgroup. As already said in the introduction, ${\mathcal B}_{g}$ is the union of the sets $F(g,p;r,s,t)$ for all admissible tuples $(g,p;r,s,t)$ and $p$ prime.

\begin{rema}[{\it Quasiconformal deformation space of a cyclic-Schottky group}]\label{obs:dimension}
As a consequence of Theorem \ref{teo1}, the dimension of the quasiconformal deformation space of a cyclic-Schottky group of type $(g,p;t,r,s)$ equals to $(3g-3-r(p-3))/p$. This comes from the fact that a loxodromic transformation is determined by its two fixed points and its multiplier, an elliptic transformation of order $p$ is determined by its two fixed points and a choice of angle of the form $2q\pi/p$, where $q \in \{1,2,\ldots,p-1\}$. In particular, each irreducible component of $F(g,p;t,r,s)$ has dimension $(3g-3-r(p-3))/p$. A model of this quasiconformal deformation space, similar to the marked Schottky space, can be obtained as a consequence of Theorem \ref{teo1}; as we do not need it, we do not describe it here.
\end{rema}

\section{Proof of Theorem \ref{topo}}
Let us consider an admissible tuple $(g,p;t,r,s)$  and let $K$ be a cyclic-Schottky group of type $(g,p;t,r,s)$. By Theorem \ref{teo1}, we know that $K$ is 
constructed as the free product of 
$``t"$ cyclic groups generated by loxodromic transformations $A_{1},\dots,A_{t}$,
$``r"$ cyclic groups generated by elliptic transformations $E_{1},\dots,E_{r}$, each one of them of order $p$,  and $``s"$ abelian groups $H_{1},\dots,H_{s}$, where $H_{j} \cong {\mathbb Z} \oplus {\mathbb Z}_{p}$ is generated by a loxodromic transformation $T_{j}$ and an elliptic transformation $F_{j}$ of order $p$, so that $T_{j}F_{j}=F_{j}T_{j}$. By Remark \ref{unicidad}, the values of $t$, $r$ and $s$ are uniquely determined by $K$.

By Klein-Maskit Combination Theorems, a geometric automorphism of $K$ can only permute the generators $E_{1},\dots,E_{r}$ (up to conjugation by some element of $K$ and inversion) and can only permute the generators $F_{1},\dots,F_{s}$ (up to conjugation by some element of $K$ and inversion).
 
The normal subgroups of index $p$ of $K$ which are Schottky groups of rank $g$ are obtained as the kernel of a surjective homomorphism $\Phi:K \to \langle a \rangle \cong {\mathbb Z}_{p}$ with a torsion-free kernel. The torsion free condition is equivalent to have $\Phi(E_{j}),\Phi(F_{k}) \in \{a,a^{2},\dots,a^{p-1}\}$. Clearly, $\Phi$ is uniquely determined by its kernel, up to post-composition by an automorphism of ${\mathbb Z}_{p}$. 

Let $\Gamma_{1}, \Gamma_{2}$ be two Schottky groups of rank $g$, both of which are normal subgroups of index $p$ in $K$ and let $\Phi_{j}:K \to {\mathbb Z}_{p}$ be surjective homomorphisms with $\ker(\Phi_{j})=\Gamma_{j}$. Then $\Psi(\Gamma_{1})=\Gamma_{2}$ for a geometric automorphism $\Psi$ of $K$ if and only if  $A \circ \Phi_{1}=\Phi_{2} \circ \Psi$, for some $A \in {\rm Aut}({\mathbb Z}_{p}) \cong {\mathbb Z}_{p-1}$; in which case  we say that the surjective homomorphisms $\Phi_1$ and $\Phi_2$ are equivalent.
To obtain the desired result, we only need to count how many torsion-free kernel surjective homomorphisms $\Phi:K \to \langle a \rangle$, up to equivalence, are there.

\subsubsection*{Step 1: Starting data}
Let us start with a surjective homomorphism $\Phi:K \to {\mathbb Z}_{p}=\langle a: a^{p}=1\rangle$ whose kernel is torsion-free. As already observed above, the torsion-free condition implies that $\langle \Phi(E_{j})\rangle=\langle a\rangle=\langle \Phi(F_{k})\rangle$, for every $j=1,\dots,r$ and $k=1,\dots,s$.

\subsubsection*{Step 2: Applying some geometric automorphisms of $K$}
If $s>0$, then (by replacing each loxodromic generator $T_{k}$ by a new loxodromic transformation $F_{k}^{n_{k}}T_{k}$, for a suitable $n_{k}$), we may assume that $\Phi(T_{k})=1$. Note that this replacement is provided by a geometric automorphism of $K$.
Similarly, if $r>0$ or $s>0$, then (by replacing the loxodromic generator $A_{j}$ by a new loxodromic $E_{1}^{n_{j}}A_{j}$ or $F_{1}^{n_{j}}A_{j}$, for suitable $n_{j}$) we can assume that $\Phi(A_{j})=1$. This change is again produced by a geometric automorphism of $K$.
If $r=s=0$, then the surjectivity of $\Phi$ asserts there is some $j \in \{1,\dots,t\}$ so that 
$\Phi(A_{j}) \neq 1$. Now, by replacing the loxodromic generator $A_{i}$ ($i \neq j$) by the new loxodromic $A_{j}^{u_{i}}A_{i}$, for suitable $u_{i}$, we may assume that $\Phi(A_{i})=1$. This change is also produced by a geometric automorphism of $K$.

\subsubsection*{Step 3: A new equivalent homomorphism}
By Step 2, we may replace our starting homomorphism $\Phi$ with an equivalent one (which we still calling $\Phi$) so that:

\noindent
(1) if $r+s>0$, then $\Phi(A_{j})=\Phi(T_{k})=1$, for every $j=1,\dots,t$ and $k=1,\dots,s$; 

\noindent
(2) if $r=s=0$, then  $\Phi(A_{1}) \neq 1$ and $\Phi(A_{j})=1$, for every $j=2,\dots,t$.

\subsubsection*{Step 4: The case $r+s>0$}
The surjective homomorphisms $\Phi$ satisfies that  
$\Phi(A_{j})=\Phi(T_{k})=1$, for every $j=1,\dots,t$ and $k=1,\dots,s$. 
Let $\Phi(E_{j})=a^{u_{j}}$ and $\Phi(F_{k})=a^{v_{k}}$, where $u_{j},v_{k} \in \{1,2,\dots,p-1\}$, for each $j=1,\dots,r$ and each $k=1,\dots,s$. A post-composition, by a geometric automorphism of $K$, will change the tuples $(u_{1},\dots,u_{r})$ and $(v_{1},\dots,v_{s})$ into tuples $(\widehat{u}_{1},\dots,\widehat{u}_{r})$ and $(\widehat{v}_{1},\dots,\widehat{v}_{s})$, where $\widehat{u}_{j} \in \{u_{\sigma(j)},p-u_{\sigma(j)}\}$ for some permutation $\sigma \in {\mathfrak S}_{r}$ and $\widehat{v}_{k} \in \{v_{\eta(k)},p-v_{\eta(k)}\}$ for some permutation $\eta \in {\mathfrak S}_{s}$. Similarly as in \cite{Ka-Mi}, one obtains that the number of different surjective homomorphisms, up to pre-composing  by a geometric automorphism of $K$, is $$\left( \begin{array}{c}
r+(p-3)/2 \\
(p-3)/2
\end{array}
\right)
\left( \begin{array}{c}
s+(p-3)/2 \\
(p-3)/2
\end{array}
\right)
$$

\subsubsection*{Step 5: The case $r=s=0$}
We only have one possibility given by the surjective homomorphism $\Phi$ so that 
$\Phi(A_{j})=1$, for every $j=2,\dots,t$, and $\Phi_{1}(A_{1})=a$.

\section{Proof of Theorem \ref{maintheo}}

\subsection{}
Before providing the proof of Theorem \ref{maintheo}, we need a couple of remarks and the definition of some maps needed in the proof.

\begin{rema}\label{rem6}
If $K_{1}$ and $K_{2}$ are two finitely generated Kleinian groups that are quasiconformal conjugated, and 
$W:\widehat{\mathbb C} \to \widehat{\mathbb C}$ is a quasiconformal homeomorphism such that $K_{2}=W K_{1} W^{-1}$, then there is a natural holomorphic isomorphism 
$$\iota_{K_{1},K_{2},W}:{\mathcal Q}(K_{1}) \to {\mathcal Q}(K_{2})$$
defined as follows (see \cite{Nag} for details). If $[\mu] \in {\mathcal Q}(K_{1})$ and $W_{\mu}$ is a 
quasiconformal homeomorphism with complex dilation $\mu$, then we set $\iota_{K_{1},K_{2},W}([\mu])=[\nu] \in {\mathcal Q}(K_{2})$, where $[\nu]$ is the class of a Beltrami coefficient of the quasiconformal homeomorphism $W_{\mu} \circ W^{-1}$.
The quasiconformal homeomorphism $W$ conjugating $K_{1}$ onto $K_{2}$ above is not unique, but any other is of the form $W \circ W_{1}^{-1}$, where $W_{1}$ is a quasiconformal homeomorphism with $W_{1} K_{1} W_{1}^{-1}=K_{1}$. The quasiconformal homeomorphism $W_{1}$ induces (as above) a holomorphic automorphism $\iota_{K_{1},K_{1},W_{1}}$ of ${\mathcal Q}(K_{1})$. In this way, the isomorphism $\iota_{K_{1},K_{2},W}$ is unique up to pre-composition with a holomorphic automorphism of ${\mathcal Q}(K_{1})$, that is, we obtain a well defined bijection
$\widehat{\iota}_{K_{1},K_{2}}:{\mathcal M}(K_{1}) \to {\mathcal M}(K_{2})$ making the following diagram commutative.
{\small
$$
\begin{CD}
{\mathcal Q}(K_{1}) @>\iota_{K_{1},K_{2},W}>> {\mathcal Q}(K_{2})\\
@VV\pi_{K_{1}}V @VV\pi_{K_{2}}V\\
{\mathcal M}(K_{1}) @>\widehat{\iota}_{K_{1},K_{2}}>> {\mathcal M}(K_{2})
\end{CD}
$$
}
\end{rema}

\begin{rema}\label{incrusta}
If $\Gamma$ is a finite index subgroup of $K$, then each Beltrami coefficient for $K$ is also a Beltrami differential for $\Gamma$ and, as both $K$ and $\Gamma$ have the same limit set (because of the finite index property) if two Beltrami coefficients for $K$ are equivalent with respect to $K$, then they are also equivalent with respect to $\Gamma$. In particular, it provides a natural holomorphic embedding
$\iota_{\Gamma}:{\mathcal Q}(K) \hookrightarrow {\mathcal Q}(\Gamma): [\mu] \mapsto [\mu].$
\end{rema}

\subsection{}
Let us fix a cyclic-Schottky group $K$ of type $(g,p;t,r,s)$, where $p \geq 3$ is a prime integer. If we set $m=M(g,p;t,r,s)$, then Theorem \ref{topo} ensures that inside $K$ there are exactly $m$ Schottky groups of rank $g$, say $\Gamma_{1}, \dots,\Gamma_{m}$, each one a normal subgroup of index $p$, no two of them being geometrically equivalent. 
The inclusion of $\Gamma_{j}$ inside $K$ induces a holomorphic map $\chi \circ \iota_{j}:{\mathcal Q}(K) \to {\mathcal MS}_{g}$, where $\iota_{j}=\iota_{\Gamma_{j}}:{\mathcal Q}(K) \to {\mathcal Q}(\Gamma_{j})$ is the embedding given in Remark \ref{incrusta}, and $\chi:{\mathcal Q}(\Gamma_{j}) \to {\mathcal MS}_{g}$ is the identification defined in Section \ref{identifica}. The image under $\pi:{\mathcal MS}_{g} \to {\mathcal S}_{g}$, defined in Section \ref{identifica}, of the space
$\chi(\iota_{j}({\mathcal Q}(K))) \subset {\mathcal MS}_{g}$ is a connected subset $F_{j}(g,p;t,r,s)$ of $F(g,p;t,r,s) \subset {\mathcal S}_{g}$. In this way, $$F(g,p;t,r,s)=F_{1}(g,p;t,r,s) \cup \cdots \cup F_{m}(g,p;t,r,s).$$

By Theorem \ref{topo}, the number of connected components of $F(g,p;t,r,s)$ is bounded above by $M(g,p;t,r,s)$. Since $M(g,3;t,r,s)=1$, and $F(g,2;t,r,s)$ is connected \cite{DGGH}, part (1) is also obtained. Let us assume, from now on, that $p \geq 5$. To prove part (3), let us start with the following

\begin{lemm}\label{lemap5}
If $p \geq 5$ is a prime, then $F_{j_{1}}(g,p;t,r,s) \cap F_{j_{2}}(g,p;t,r,s) \neq \emptyset$ if and only if 

\noindent
(1) there are two different cyclic-Schottky groups of type $(g,p;t,r,s)$, say $K_{1}$ and $K_{2}$, containing the same Schottky group $\Gamma$ of rank $g$ as index $p$ normal subgroup and

\noindent
(2) for $i \in \{1,2\}$ there is a geometric isomorphism $\psi_{i}:K \to K_{i}$ with $\psi_{i}(\Gamma_{j_{i}})=\Gamma$  (in other words, $\Gamma$ looks like $\Gamma_{j_{1}}$ inside $K_{1}$ and like $\Gamma_{j_{2}}$ in $K_{2}$).

\end{lemm}
\begin{proof} 
The condition 
$F_{j_{1}}(g,p;t,r,s) \cap F_{j_{2}}(g,p;t,r,s) \neq \emptyset$
is equivalent to have a quasiconformal homeomorphisms $W_{1}, W_{2}:\widehat{\mathbb C} \to \widehat{\mathbb C}$ so that $W_{i} K W_{i}^{-1} < {\mathbb M}$, for $i=1,2$, and $W_{1} \Gamma_{j_{1}} W_{1}^{-1}=A W_{2} \Gamma_{j_{2}} W_{2}^{-1} A^{-1}$, for some $A \in {\mathbb M}$. We take $K_{1}=W_{1}KW_{1}^{-1}$, $K_{2}=AW_{2}KW_{2}^{-1}A^{-1}$, $\Gamma=W_{1}\Gamma_{j_{1}}W_{1}^{-1}$, $\psi_{1}$ is conjugation by $W_{1}$ and $\psi_{2}$ is conjugation by $AW_{2}$.
\end{proof}

\begin{rema}
Let $p \geq 5$ be a prime integer.

\noindent
(1) As a consequence of Lemma \ref{lemap5}, the number of connected components of $F(g,p;t,r,s)$ is strictly smaller than $M(g,p;t,r,s)$ (the number of the irreducible components) exactly when there are two different cyclic-Schottky groups of the same type $(g,p;t,r,s)$, say $K_1$ and $K_2$, each one containing a common Schottky group $\Gamma$ of rank $g$ as a normal subgroup of index $p$, for which there is no an isomorphism (as abstract groups) $\psi:K_1 \to K_2$ preserving $\Gamma$. As noticed before, this only happens if $r$ or $s$ is different from zero. Let $\Omega$ be the region of discontinuity of $\Gamma$ and $M_{\Gamma}=({\mathbb H}^{3} \cup \Omega)/\Gamma$ be the uniformized handlebody by $\Gamma$. If $H_{j}=K_{j}/\Gamma \cong {\mathbb Z}_p$, then $H_1, H_2 \leq {\rm Aut}(M_{\Gamma})$ cannot be topologically conjugated. In fact, if there is a homeomorphism $f:M_{\Gamma} \to M_{\Gamma}$ conjugating $H_1$ into $H_2$, then we may lift  $f$ to the universal cover space ${\mathbb H}^{3} \cup \Omega$ to obtain a homeomorphism preserving $\Gamma$ and conjugating $K_1$ into $K_2$, a contradiction.

\noindent
(2) As noticed in Remark \ref{obs2}, in \cite{L-H} it was shown the existence of a prime integer $p_{0}$ (depending on $t+s$ and $r$) so that for $p \geq p_{0}$ prime integer the groups $H_1$ and $H_2$ must coincide, that is $K_1=K_2$; so the number of connected components of $F(g,p;t,r,s)$ equals $M(g,p;t,r,s)$. Also, if $t=s=0$, then the results in \cite{Gabino}, together (1) above, assert that the number of connected components of $F(g,p;t,r,s)$ equals $M(g,p;t,r,s)$. 

\end{rema}


Let $K_{1}$ and $K_{2}$ be two different cyclic-Schottky groups of type $(g,p;t,r,s)$, both containing the same Schottky group $\Gamma$ of rank $g$ as an index $p$ normal subgroup. Assume the existence of two geometric isomorphisms $\psi_{i}:K \to K_{i}$ with $\psi_{i}(\Gamma_{j_{i}})=\Gamma$, for $i=1,2$.
Let $\Omega$ be the region of discontinuity of $\Gamma$. Then the finite index property of $\Gamma$ in $K_{j}$ asserts that $\Omega$ is also de region of discontinuity of $K_{1}$, $K_{2}$ and $\langle K_{1},K_{2}\rangle$. The manifold $M_{\Gamma}=({\mathbb H}^{3} \cup \Omega)/\Gamma$ is a handlebody of genus $g$. Each $K_{j}$ induces the cyclic group $H_{j}=K_{j}/\Gamma \cong {\mathbb Z}_{p}$ of automorphisms of $M_{\Gamma}$. Since $\Gamma_{j_{1}}$ and $\Gamma_{j_{2}}$ are geometrically non-equivalent, there is no M\"obius transformation $B \in {\mathbb M}$ satisfying that $B \Gamma B^{-1}=\Gamma$ and $B K_{1} B^{-1}=K_{2}$. In particular, the cyclic groups $H_{1}$ and $H_{2}$ are non-conjugated in the group ${\rm Aut}(M_{\Gamma})$ of automorphisms of $M_{\Gamma}$. So, we may assume that $H_{1}$ and $H_{2}$ belong to the same $p$-Sylow subgroup of ${\rm Aut}(M_{\Gamma})$ and that they are in fact different. Let us consider the $p$-group $G=\langle H_{1}, H_{2} \rangle$, which has order $p^{\alpha}$, where $ \alpha \geq 2$. In particular, there is a subgroup $J<G$ of order $p^{2}$ with $H_{1} \lhd J$. Since $J$ must be abelian, it may happen that $J \cong {\mathbb Z}_{p^{2}}$ or $J \cong {\mathbb Z}_{p} \oplus {\mathbb Z}_{p}$. The orbifold ${\mathcal O}_{1}=M_{\Gamma}/H_{1}$ is a handlebody of genus $t+s$, whose conical locus is formed by $r$ simple arcs and $s$ simple loops (all of them pairwise disjoint) each one of order $p$. The group $J$ induces a cyclic group $\widehat{J}$ acting on ${\mathcal O}_{1}$ which keeps invariant the conical locus. Since $p \geq 5$, none of the loops or arcs can be kept invariant under $\widehat{J}$. It follows that $r$ and $s$ are multiples of $p$; so part (3) of Theorem \ref{maintheo} is proved.

\section{Examples} \label{Sec:ejemplo}
\subsection{Example 1}\label{Ejemplo1}
As $M(g,p;t,0,0)=1$, the locus $F(g,p;t,0,0)$ is connected. Let us assume that $g=26$, $p=5$ and $t=6$.
Let $K$ be a Schottky group of rank two, say generated by $A$ and $B$, and consider the surjective homomorphism
$\theta:K \to \langle a,b: a^{5}=b^{5}=aba^{-1}b^{-1}=1\rangle \cong {\mathbb Z}_{5}^{2}$, defined by 
$A \mapsto a, \quad B \mapsto b$. The kernel of $\theta$ is a Schottky group $\Gamma$ of rank $26$, the group $K_{1}=\theta^{-1}(\langle a \rangle)$ is a cyclic-Schottky group of type $(26,5;6,0,0)$, generated by $C_{1}=A$, $C_{2}=BAB^{-1}$, $C_{3}=B^{2}AB^{-2}$, $C_{4}=B^{3}AB^{-3}$, $C_{5}=B^{4}AB^{-4}$, $C_{6}=B^{5}$, and the group $K_{2}=\theta^{-1}(\langle b \rangle)$ is a cyclic-Schottky group of type $(26,5;6,0,0)$, generated by $D_{1}=B$, $D_{2}=ABA^{-1}$, $D_{3}=A^{2}BA^{-2}$, $D_{4}=A^{3}BA^{-3}$, $D_{5}=A^{4}BA^{-4}$, $D_{6}=A^{5}$.
The group $\Gamma$, seen inside $K_{1}$ is generated by $C_{1}^{5}$, $C_{1}C_{6}C_{1}^{-1}$, $C_{1}^{2}C_{6}C_{1}^{-2}$, $C_{1}^{3}C_{6}C_{1}^{-3}$, $C_{1}^{4}C_{6}C_{1}^{-4}$, and  seen inside $K_{2}$ is generated by $D_{1}^{5}$, $D_{1}D_{6}D_{1}^{-1}$, $D_{1}^{2}D_{6}D_{1}^{-2}$, $D_{1}^{3}D_{6}D_{1}^{-3}$, $D_{1}^{4}D_{6}D_{1}^{-4}$.
The isomorphism $\psi:K \to K$, defined by $\psi(A)=B$ and $\psi(B)=A$, satisfies that  $\psi(K_{1})=K_{2}$ (in fact, $\psi(C_{j})=D_{j}$). (Up to a quasiconformal deformation of $K$, we may assume that $\psi$ is defined by conjugation by a M\"obius transformation of order two.) 
Since $\psi(C_{1}^{5})=D_{1}^{5}$, $\psi(C_{1}C_{6}C_{1}^{-1})=D_{1}D_{6}D_{1}^{-1}$, 
$\psi(C_{1}^{2}C_{6}C_{1}^{-2})=D_{1}^{2}D_{6}D_{1}^{-2}$, $\psi(C_{1}^{3}C_{6}C_{1}^{-3})=D_{1}^{3}D_{6}D_{1}^{-3}$ and $\psi(C_{1}^{4}C_{6}C_{1}^{-4})=D_{1}^{4}D_{6}D_{1}^{-4}$, one sees that $\psi$ preserves $\Gamma$. In other words, $\psi$ induces an automorphism of order two in the handlebody $M_{\Gamma}$ which conjugates $H_{1}=K_{1}/\Gamma$ into $H_{2}=K_{2}/\Gamma$.

\subsection{Example 2}\label{ejemplo2}
Let $p \geq 5$ be a prime integer and let $m \geq 4$ be some integer. Consider any two tuples 
$(\alpha_{1},\ldots,\alpha_{m}), (\beta_{1},\ldots,\beta_{m}) \in ({\mathbb Z}_{p}^{*})^{m} ={\mathbb Z}_{p-1}^{m}$
that do not belong to the same orbit under the action of ${\mathbb Z}_{p}^{*} \rtimes {\mathfrak S}_{m}$, where ${\mathbb Z}_{p}^{*}$ acts by multiplication on the coordinates and the symmetric group ${\mathfrak S}_{m}$ acts by permutation of the coordinates. Choose pairwise different complex  numbers 
$a_{1,1}$, $a_{1,2}$, $a_{2,1}$, $a_{2,2},\ldots,a_{m,1}$, $a_{m,2}$, $b_{1,1},b_{1,2},b_{2,1},b_{2,2},\ldots,b_{m,1},b_{m,2}$ and 
consider the Riemann surface (a fiber product of two $p$-gonal curves)
$$S= \left\{ \begin{array}{c}
y_{1}^{p}=\prod_{j=1}^{m} (x-a_{j,1})^{\alpha_{j}}(x-a_{j,2})^{p-\alpha_{j}}\\
y_{2}^{p}=\prod_{j=1}^{m} (x-b_{j,1})^{\beta_{j}}(x-b_{j,2})^{p-\beta_{j}}
\end{array}
\right.
$$

If $\omega_{p}=e^{2 \pi i /p}$, then  
$A_{1}(x,y_{1},y_{2})=(x,\omega_{p} y_{1}, y_{2})$ and $A_{2}(x,y_{1},y_{2})=(x,y_{1}, \omega_{p}  y_{2})$
are commuting automorphism of order $p$ on $S$.
The map $\pi:S \to \widehat{\mathbb C}: (x,y_1,y_2) \mapsto x$ is a branched Galois covering with deck group $\langle A_{1},A_{2}\rangle \cong {\mathbb Z}_{p}^{2}$ whose branch values (each one of order $p$) are given by 
$a_{1,1},a_{1,2}, a_{2,1},a_{2,2},\ldots,a_{m,1},a_{m,2}$, $b_{1,1},b_{1,2},b_{2,1},b_{2,2},\ldots,b_{m,1},b_{m,2}$.
It follows from the Riemann-Hurwitz formula that $S$ has genus $g=(p-1)(2mp-p-1)$.
The fixed points of $A_1$ are of the form 
{\small
$$u_{j,\delta,k}=\left(a_{j,\delta},0,\omega_{p}^{k} \sqrt[p]{\prod_{i=1}^{m}(a_{j,\delta}-b_{i,1})^{\beta_{i}}(a_{j,\delta}-b_{i,2})^{p-\beta_{i}}}\;\right),$$
}
and those of $A_2$ are of the form
{\small
$$v_{j,\delta,k}=\left(b_{j,\delta},\omega_{p}^{k} \sqrt[p]{\prod_{i=1}^{m}(b_{j,\delta}-a_{i,1})^{\alpha_{i}}(b_{j,\delta}-a_{i,2})^{p-\alpha_{i}}},0\right),$$
}
where $j=1,\ldots,m$, $k=0,1,\ldots,p-1$ and  $\delta=1,2$.

Observe that the rotation number of $A_1$ on each $u_{j,\delta,k}$, with $j$ and $\delta$ fixed, is the same. Similarly, the rotation number of $A_2$ on each $v_{j,\delta,k}$, with $j$ and $\delta$ fixed, is the same. By results in \cite{Hidalgo:SchottkyAuto}, there is a Schottky uniformization of $S$ (say given by the Schottky group $\Gamma$) for which $\langle A_1,A_2\rangle$ lifts; that is, there is a Kleinian group $K$ containing $\Gamma$ as a normal subgroup and $K/\Gamma \cong {\mathbb Z}_{p}^{2}$. Let $\theta:K \to \langle A_1,A_2\rangle$ be the canonical surjective homomorphism. Then $K_j=\theta^{-1}(\langle A_j \rangle$ is a Schottky-cyclic of type $(g,p;t,mp,0)$, where $t=(p-1)(m-1)$.
We claim that $\langle A_1 \rangle$ and $\langle A_2\rangle$ are topologically non-conjugated. In fact, the existence of a homeomorphism $F:S \to S$ conjugating $\langle A_1 \rangle$ into $\langle A_2\rangle$ must preserve the rotation numbers (up to the action of ${\mathbb Z}_{p}^{*} \rtimes {\mathfrak S}_{m}$). This is a contradiction with the choice of the tuples we have done. This asserts that two different irreducible components of $F(g,p;(p-1)(m-1),mp,0)$ intersect, so the number of connected components will be smaller than $M(g,p;(p-1)(m-1),mp,0)$. In fact, it can be seen (working with all possible situations) $that F(g,p;(p-1)(m-1),mp,0)$ is connected.

\section{A final remark: Counting topologically non-equivalent cyclic-Schottky groups}\label{Sec:consecuencias}
In this last section, for each integer $g \geq 2$ and each prime integer $p \geq2$, we are interested in finding the number 
$N(p,g)$ of topologically different $p$-cyclic-Schottky strata are for a fixed genus $g$.
 It is not clear, at least for the authors, a closed form for $N(p,g)$. Below, we proceed to obtain it for $p \in \{2,3\}$ and, for the case $p \geq 5$, we describe a simple algorithm.
By Theorem \ref{teo1} (see also Remark \ref{unicidad}), $N(p,g)$ is equal to the number of different admissible tuples $(g,p;t,r,s)$, so 
we only need to compute the number of different triples  $(t,r,s),$ where $t,r,s \geq 0$ and $g-1=p(t+r+s-1)-r=p(t+s-1)+r(p-1)$. Thus, 
$N(p,g)$ is exactly the number of pairs $(t,s)$ satisfying: 
\begin{enumerate}
\item[(i)] $t,s \in \{0,1,...,[(g+p-1)/p]\}$, 
\item[(ii)] $t+s \leq [(g+p-1)/p]$ and 
\item[(iii)] $g-p(t+s) \equiv 0 \mod(p-1)$.
\end{enumerate}

The above permits us to observe the following.
\begin{enumerate}
\item If $p=2$, then every pair $(t,s)$ with $t+s \in \{0,1,...,[(g+1)/2]$ satisfies (i)-(iii), in particular
{\small $$N(2,g)= \frac{(1+[\frac{g+1}{2}])(2+[\frac{g+1}{2}])}{2}.$$}

\item 
If $p=3$, then condition (iii)  is equivalent to $g-3(t+s)$ being even, so  
{\small
$$N(3,g)=\left\{ \begin{array}{cc}
\frac{\left([\frac{g+2}{3}]+2\right)^{2}}{4}, & \mbox{if $g$ is even and } \; [\frac{g+2}{3}] \; \mbox{is even}\\
\frac{\left([\frac{g+2}{3}]+1\right)^{2}}{4}, & \mbox{if $g$ is even and } \; [\frac{g+2}{3}] \; \mbox{is odd}\\
\frac{[\frac{g+2}{3}]\left([\frac{g+2}{3}]+2\right)}{4}, & \mbox{if $g$ is odd and } \; [\frac{g+2}{3}] \; \mbox{is even}\\
\frac{\left([\frac{g+2}{3}]+1\right)\left([\frac{g+2}{3}]+3\right)}{4}, & \mbox{if $g$ is odd and } \; [\frac{g+2}{3}] \; \mbox{is odd}\\
\end{array}
\right.
$$
}
\end{enumerate}

\subsection{An algorithm}
If $p \geq 5$ is a prime, then $N(p,g)$ and the corresponding triples $(t,r,s)$
can be obtained with the following short program, written for {\sc Mathematica} \cite{Mathematica}.

\begin{itemize}
\item[] $n[p\_, g\_ ] := Block[\{ \}$, $k := 0;$
Do[ Do[ Do[ If[$t+r+s - 1 \neq  0$,
\item[] If[$(g + r - 1)/(t+r+s - 1) == p$,
$\{k = k + 1,\\
{\rm Print}["(", k, ") ", "t=", t, "," , " r= ", r, "," , " s=", s, ",", " p=", p]\}]]$, 
\item[] $\{r, 0, (g + 1 - 2*(t+s))\}], \{s, 0, g\}], \{t, 0, g\}]$;
Print$["N(", p, ",", g, ")=", k]]$
\end{itemize}

For instance, 
\begin{enumerate}
\item[(i)] $N(5,5)=2$ and $(t,r,s) \in \{(0,1,1), (1,0,1)\}$.

\item[(ii)] $N(5,10)=3$ and $(t,r,s) \in \{(0,2,1),(1,1,1),(2,0,1)\}$.

\item[(iii)] $N(11,10)=1$ and $(t,r,s)=(0,0,2)$.

\item[(iv)] $N(11,100)=12$
and 
$$
(t,r,s) \in \left\{
\begin{array}{c} 
(0,0,11),  
(0,10,0), 
(1,9,0),
(2,8,0) ,
(3,7,0), 
(4,6,0),\\
(5,5,0),
(6,4,0),
(7,3,0),
(8,2,0), 
(9,1,0),
(10,0,0) 
\end{array}
\right\}.
$$
\item[(v)] $N(13,157)=16$ and
$$
(t,r,s) \in \left\{
\begin{array}{c} 
(0,1,13),
(0,13,0),
(1,0,13),
(1,12,0),
(2,11,0),
(3,10,0),
(4,9,0),
(5,8,0),\\
(6,7,0),
(7,6,0),
(8,5,0),
(9,4,0),
(10,3,0),
(11,2,0),
(12,1,0),
(13,0,0)
\end{array}
\right\}.
$$

\end{enumerate}


\end{document}